\documentclass[10pt]{article}

\usepackage[utf8]{inputenc}
\usepackage[T1]{fontenc}
\usepackage{amsmath,amssymb,amsthm,mathtools}
\usepackage{enumitem}
\usepackage[skip=0.5\baselineskip plus 2pt minus 1pt,indent=0pt]{parskip}
\usepackage{hyperref}
\usepackage[margin=1.15in]{geometry}

\hypersetup{colorlinks=true,linkcolor=blue,citecolor=blue,urlcolor=blue}

\newtheorem{theorem}{Theorem}[section]
\newtheorem{lemma}[theorem]{Lemma}
\newtheorem{proposition}[theorem]{Proposition}
\newtheorem{corollary}[theorem]{Corollary}

\theoremstyle{definition}
\newtheorem{definition}[theorem]{Definition}
\theoremstyle{remark}

\DeclareMathOperator{\ex}{ex}
\DeclareMathOperator{\dist}{dist}

\newcommand{\Fstar}{F^\star}
\newcommand{\calI}{\mathcal I}
\newcommand{\calB}{\mathcal B}
\newcommand{\calF}{\mathcal F}
\newcommand{\calA}{\mathcal A}

\newcommand{\calC}{\mathcal C}

\title{A single $3$-graph with infinite stability number}
\author{Heng~Li\thanks{School of Mathematics, Shandong University, Jinan, China, and Extremal Combinatorics and Probability Group, Institute for Basic Science, Daejeon, South Korea.  Email: \texttt{heng.li@sdu.edu.cn}.}
\and
Xizhi~Liu\thanks{School of Mathematical Sciences, University of Science and Technology of China, Hefei, China. Email:~\texttt{liuxizhi@ustc.edu.cn}.}}
\date{\today}

\begin{document}
\maketitle

\begin{abstract}
    The stability number of a forbidden family measures how many different structures are needed to approximate all near-extremal constructions avoiding it.
    An infinite stability number means that no finite list of structures suffices.  We construct a simple explicit $3$-graph whose stability number is infinite.
    This extends the infinite-stability phenomenon for finite forbidden families, established by Hou--Li--Liu--Mubayi--Zhang, to the single-forbidden setting, and further develops the single-$3$-graph direction of Balogh--Clemen--Luo, in which exponentially many exact extremal constructions coexist with stability.
\end{abstract}

\section{Introduction}

The Tur\'an problem goes back to Tur\'an's 1941 theorem \cite{Turan1941}, which determines the maximum number of edges in an $n$-vertex graph containing no copy of $K_t$ and identifies the balanced complete $(t-1)$-partite graph as the unique extremal construction. For $r\ge 2$, an \emph{$r$-uniform hypergraph}, or \emph{$r$-graph}, $H$ is a collection of $r$-subsets of a finite set $V(H)$. We identify a hypergraph with its edge set, and hence write $|H|$ for the number of edges of $H$. Given a family $\calF$ of $r$-graphs, the \emph{Tur\'an number} $\ex(n,\calF)$ is the maximum number of edges in an $\calF$-free $r$-graph on $n$ vertices, where \emph{$\calF$-free} means containing no member of $\calF$ as a subgraph.  The \emph{Tur\'an density} is
\[
        \pi(\calF)=\lim_{n\to\infty}{\ex(n,\calF)}/{\tbinom n r},
\]
and if $\calF=\{F\}$ we write $\pi(F)$.

For graphs, the Erd\H{o}s--Stone--Simonovits theorem \cite{ErdosStone,ErdosSimonovitsLimit} gives the asymptotic Tur\'an density for every fixed forbidden graph, and the Simonovits stability theorem \cite{SimonovitsStability} shows that near-extremal graphs are close to the appropriate complete multipartite graph.  Thus graph Tur\'an problems are often governed by a small, rigid set of extremal constructions.

For hypergraphs, the situation is much less rigid; see Keevash's survey \cite{KeevashSurvey}.  The classical test case is Tur\'an's tetrahedron problem: determine the maximum size of a $3$-graph containing no copy of the complete $3$-graph $K_4^3$.  Tur\'an conjectured the density to be $5/9$, but even the conjectural extremal picture is far from unique.  Brown \cite{BrownTuran34} and Kostochka \cite{Kostochka} found further constructions for this problem, and Kostochka \cite{Kostochka} proved that, if Tur\'an's conjecture is correct, then for $n=3k$ there are at least $2^{k-2}$ non-isomorphic extremal hypergraphs; Frohmader \cite{Frohmader} later constructed exponentially many conjectured extremal examples in the remaining congruence classes.  This phenomenon suggests that hypergraph Tur\'an problems should be studied not only through their densities, but also through the number and geometry of their extremal and near-extremal models.

This paper works at the tripartite density $2/9$.  A basic example at this density is the \emph{generalized triangle}
\[
        F_5=\{123,124,345\}.
\]
Bollob\'as \cite{Bollobas} proved the extremal theorem for cancellative triple systems, Frankl and F\"uredi \cite{FranklFuredi} proved $\pi(F_5)=2/9$ and the exact theorem for all sufficiently large $n$, and Keevash--Mubayi \cite{KeevashMubayi} gave stability refinements for cancellative hypergraphs.  The same tripartite structure also appears in the enumeration theorem of Balogh and Mubayi \cite{BaloghMubayi}, for almost all $F_5$-free triple systems.

A \emph{finite-template form of stability} was introduced by Mubayi \cite{MubayiTriangleFree} and, independently, by Pikhurko \cite{PikhurkoGeneralizedTriangle} to describe Tur\'an problems in which all near-extremal examples can be approximated by one of finitely many model constructions.
\begin{definition}[\emph{stability number}]\label{def:t-stable}
Let $\calF$ be a family of $r$-graphs. We say that $\calF$ is \emph{$t$-stable} if, for every $m\in\mathbb N$, there exist $r$-graphs $G_1(m),\ldots,G_t(m)$ on $m$ vertices such that the following holds. For every $\delta>0$ there exist $\epsilon>0$ and $N_0$ such that, whenever $n\ge N_0$ and $H$ is an $n$-vertex $\calF$-free $r$-graph with $|H|>(1-\epsilon)\ex(n,\calF)$, after relabelling, the hypergraph $H$ can be transformed into some $G_i(n)$ by adding and deleting at most $\delta n^r$ edges. The least such integer $t$ is the \emph{stability number} of $\calF$ and is denoted by $\xi(\calF)$; if no such $t$ exists, we put $\xi(\calF)=\infty$.
\end{definition}

After these formulations, Liu and Mubayi \cite{LiuMubayiNoStability} constructed the first finite 3-graph family that is not stable in the ordinary one-model sense, and proved that it is nevertheless $2$-stable.  Liu, Mubayi and Reiher \cite{LiuMubayiReiherUnified,LiuMubayiReiherMany} developed a unified stability framework and constructed, for every fixed $t$, finite families of triple systems with stability number $t$.  Hou, Li, Liu, Mubayi and Zhang \cite{HLLMZ} then proved exact and stability results with infinitely many extremal constructions, using their crossed blowup operation.  Related recent work includes the $2$-stable family of Zhang, Hou and Li \cite{ZhangHouLi2Stable}, the mixing-pattern constructions of Liu and Pikhurko \cite{LiuPikhurko}, and the $1$-stable examples of Balogh, Clemen and Luo \cite{BaloghClemenLuo} of single non-degenerate 3-graphs with exponentially many exact extremal constructions.

Although finite forbidden families are indispensable in this theory, the single-forbidden case is the basic unit of the classical Tur\'an problem.  The tetrahedron problem already points in this direction: Kostochka's constructions \cite{Kostochka} were used by Liu and Mubayi \cite{LiuMubayiNoStability} to show that Tur\'an's conjecture $\pi(K_4^3)=5/9$ would imply $\xi(K_4^3)=\infty$.  Thus it is natural to ask whether this single-forbidden infinite-stability phenomenon can be proved unconditionally. Our main result answers this question affirmatively.

Let
\[
        K_4^-:=\{abc,abd,acd\}
        \qquad\text{and}\qquad
        \Fstar:=\{123,124,345,156,257\}.
\]
Set
\[
        F:=K_4^-\times \Fstar,
\]
where the product is the categorical product defined in Section~\ref{sec:preliminaries}.  Thus $F$ has $28$ vertices and $3\cdot 5\cdot 3!=90$ edges.

\begin{theorem}\label{thm:main}
The 3-graph $F$ satisfies $\pi(F)=\frac 29$ and $\xi(F)=\infty$.
\end{theorem}

The proof has three short components.  First, we prove an auxiliary two-forbidden-graph theorem with density $2/9$.  Second, a Brown--Simonovits blowup argument turns this auxiliary theorem into the upper bound for our single-forbidden graph.  Finally, we use the one-parameter family of crossed blowups from \cite{HLLMZ}.  In the language of Section~\ref{sec:stability}, these crossed blowups are rank-two cut templates.  We prove that every rank-two cut template is $F$-hom-free, while rank three already contains $F$ homomorphically.  Hence the crossed-blowup family produces asymptotically extremal $F$-free graphs that are pairwise far apart in edit distance, forcing $\xi(F)=\infty$.

\section{Preliminaries}\label{sec:preliminaries}


We first recall homomorphisms and the categorical product. A \emph{homomorphism} $\phi:G_1\to G_2$ between 3-graphs is a map $V(G_1)\to V(G_2)$ such that $\phi(e)\in G_2$ for every $e\in G_1$. We write $G_1\to G_2$ if such a homomorphism exists and $G_1\nrightarrow G_2$ otherwise.  For a single 3-graph $F$, a 3-graph $T$ is \emph{$F$-hom-free} if $F\nrightarrow T$; equivalently, every blowup of $T$ is $F$-free.  For a family $\calF$, we say that $T$ is \emph{$\calF$-hom-free} if $J\nrightarrow T$ for every $J\in\calF$.

For 3-graphs $G$ and $H$, their \emph{categorical product} $G\times H$ is the 3-graph with vertex set
\[
        V(G\times H)=V(G)\times V(H),
\]
where a triple $\{(x_1,y_1),(x_2,y_2),(x_3,y_3)\}$ is an edge if and only if
\[
        \{x_1,x_2,x_3\}\in G
        \quad\text{and}\quad
        \{y_1,y_2,y_3\}\in H.
\]
In particular, the two coordinate projections are homomorphisms $G\times H\to G$ and $G\times H\to H$.
This is the standard $3$-uniform \emph{direct product}.  For $r=2$, it is the categorical product of graphs, also called the direct, Kronecker or cardinal product, going back to \v{C}ul\'ik \cite{Culik1958} and Weichsel \cite{Weichsel1962}.  Hypergraph analogues were studied from a categorical viewpoint by D\"orfler and Waller \cite{DorflerWaller}; see also the survey of Hellmuth, Ostermeier and Stadler \cite{HellmuthOstermeierStadler}.  In the uniform Tur\'an density setting, analogous product/colorability ideas appear in the palette approach: Lamaison \cite{LamaisonPalette} proved that palettes determine uniform Tur\'an density, and Kr\'al', Ku\v{c}er\'ak, Lamaison and Tardos \cite{KralKucerakLamaisonTardos} use products and homomorphisms of palettes in their classification framework.


Let $G$ be a 3-graph with vertex set $[m]$, and let $V_1,\ldots,V_m$ be pairwise disjoint sets. The \emph{blowup} $G[V_1,\ldots,V_m]$ is obtained from $G$ by replacing every vertex $i$ by the set $V_i$ and every edge $ijk$ by the complete tripartite 3-graph with parts $V_i,V_j,V_k$. Equivalently, $J\to G$ if and only if $J$ is contained in some blowup of $G$.

A family $\calF$ of 3-graphs is \emph{blowup-invariant} if every $\calF$-free 3-graph is $\calF$-hom-free.

The \emph{Lagrange polynomial} of $G$ is
\[
        p_G(x_1,\ldots,x_m)=\sum_{ijk\in G}x_ix_jx_k.
\]
If $|V_i|=x_i n+O(1)$ for all $i$, then the blowup $G[V_1,\ldots,V_m]$ has
\[
        p_G(x)n^3+O(n^2)
\]
edges. Thus $p_G(x)$ records the normalized cubic edge count of the blowup with part proportions $x$, and the corresponding density normalized by $\binom n3$ is $6p_G(x)+o(1)$.

The \emph{Lagrangian} of $G$ is
\[
        \lambda(G)=\max\{p_G(x_1,\ldots,x_m):(x_1,\ldots,x_m)\in\Delta_{m-1}\},
\]
where $\Delta_{m-1}$ is the standard simplex. We also write
\[
        Z(G)=\{x\in\Delta_{m-1}:p_G(x)=\lambda(G)\}.
\]

For two 3-graphs $G$ and $H$ with the same number of vertices, we write
\[
        \dist(G,H):=
        \min_{\psi:V(G)\to V(H)} |\psi(G)\triangle H|
\]
for their \emph{edit distance}, where the minimum is over all bijections $\psi$.  When the vertex set is fixed, this agrees with the usual edit distance after the best relabelling.

We will use the following standard consequence of the blowup theorem of Brown and Simonovits \cite{BrownSimonovits}. If $F\to G$, then adding $G$ to a family that already contains $F$ does not change the Tur\'an density. More generally, if $F_i\to G_i$ for all $i\in[m]$ and $\calA$ is any family of 3-graphs, then
\[
        \pi(\calA\cup\{F_1,\ldots,F_m\})=\pi(\calA\cup\{F_1,\ldots,F_m,G_1,\ldots,G_m\}).
\]
We will also use \emph{Zykov symmetrization}, originating with Zykov~\cite{Zykov1949}, in the blowup-invariant form used by Hou, Li, Liu, Mubayi and Zhang~\cite{HLLMZ}: for a blowup-invariant family, an extremal hypergraph may be assumed \emph{symmetrized}.  Here symmetrized means that any two vertices not contained together in an edge are \emph{twins}, i.e. they have the same link.

For a 3-graph $H$, its \emph{shadow} is
\[
        \partial H=\{xy:\text{there exists }z\text{ with }xyz\in H\}.
\]
We write $S_3(n)$ for the complete balanced tripartite 3-graph on $n$ vertices. It has asymptotic edge density $2/9$.

\section{The Tur\'an density of \texorpdfstring{$F$}{F}}\label{sec:turan-density}

We first record the lower bound.  The product $F$ is not tripartite. Indeed, the five vertices
\[
        (a,1),\quad (d,2),\quad (c,3),\quad (b,4),\quad (a,5)
\]
span the generalized triangle $F_5$, with edges coming from $acd\times123$, $abd\times124$, and $abc\times345$.  Thus the complete balanced tripartite 3-graph $S_3(n)$ is $F$-free, and hence
\[
        \pi(F)\ge \frac29.
\]

For the upper bound, we show that the two small coordinate graphs $K_4^-$ and $\Fstar$ already determine a Tur\'an problem with density $2/9$.

\begin{lemma}\label{lem:KE-density}
$\pi(\{K_4^-,\Fstar\})=\frac 29$.
\end{lemma}

\begin{proof}
The lower bound is given by $S_3(n)$. Indeed, $K_4^-$ is not tripartite, and $\Fstar$ contains the non-tripartite 3-graph $F_5=\{123,124,345\}$. Hence $S_3(n)$ is $\{K_4^-,\Fstar\}$-free and has density $2/9$.

It remains to prove the upper bound. Let $\calI(\Fstar)$ be the finite family of all homomorphic images of $\Fstar$, and set
\[
        \calB:=\{K_4^-\}\cup \calI(\Fstar).
\]
Since $\Fstar\to J$ for every $J\in\calI(\Fstar)$, repeated application of the Brown--Simonovits blowup theorem gives $\pi(\{K_4^-,\Fstar\})=\pi(\{K_4^-\}\cup\calI(\Fstar))=\pi(\calB)$.
The family $\calB$ is \emph{blowup-invariant}. Indeed, $K_4^-$ is \emph{2-covered}, meaning that every pair of vertices is contained in an edge, so every homomorphic image of $K_4^-$ is $K_4^-$ itself. Moreover, $\calI(\Fstar)$ is closed under taking homomorphic images. Hence if a $\calB$-free 3-graph $H$ admitted a homomorphism from some member of $\calB$, then $H$ would contain a member of $\calB$, a contradiction.

By Zykov symmetrization, it suffices to consider symmetrized $\calB$-free 3-graphs. Let $H$ be such a graph, and let $T$ be its quotient obtained by taking one vertex from each equivalence class. Then $H$ is a blowup of $T$. Moreover, the graph $\partial T$ is complete: if two distinct quotient vertices were not contained together in any edge of $T$, then their representatives in $H$ would be non-adjacent in the shadow and hence twins, contradicting that they lie in distinct quotient classes.

We claim that $F_5\nrightarrow T$. Suppose first that a homomorphism $F_5\to T$ is not injective. Vertices lying in a common edge of $F_5$ cannot be identified, since the image of that edge must be a genuine 3-edge. Since the only pairs of vertices of $F_5$ that are not contained together in an edge are $\{1,5\}$ and $\{2,5\}$, the only possible identifications are $5=1$ or $5=2$. In either case, the images of \(123,124,345\) form a copy of $K_4^-$, contradicting that $T$ is $K_4^-$-free.

Thus every homomorphism $F_5\to T$ would have to be injective. In this case $T$ contains vertices, relabelled as $1,2,3,4,5$, with edges \(123,124,345\).
Since $\partial T$ is complete, the pair $15$ is contained in an edge $156$ of $T$, and the pair $25$ is contained in an edge $257$ of $T$. The vertices $6$ and $7$ are allowed to coincide with previously named vertices, provided the displayed edges are genuine triples. Consequently, $T$ contains a homomorphic image of $\Fstar$, contradicting that $T$ is $\calI(\Fstar)$-free.

Therefore $F_5\nrightarrow T$. Since $H$ is a blowup of $T$, the 3-graph $H$ is $F_5$-free. The Frankl--F\"uredi theorem \cite{FranklFuredi} gives
\[
        |H|\le \ex(n,F_5)=\left(\frac 29+o(1)\right)\binom n3.
\]
It follows that $\pi(\calB)\le 2/9$, completing the proof.
\end{proof}

\begin{corollary}\label{cor:exact-density}
$\pi(F)=\frac 29$.
\end{corollary}

\begin{proof}
The lower bound was proved at the beginning of the section.  For the upper bound, the two coordinate projections give homomorphisms $F\to K_4^-$ and $F\to \Fstar$. By the Brown--Simonovits blowup theorem,
\[
        \pi(F)=\pi(\{F,K_4^-,\Fstar\})\le \pi(\{K_4^-,\Fstar\})=\frac 29,
\]
where the last equality follows from Lemma~\ref{lem:KE-density}.
\end{proof}

\section{Cut constructions and stability}\label{sec:stability}

We now present the lower-bound constructions used for stability.  The key point is that the crossed blowups of Hou--Li--Liu--Mubayi--Zhang are exactly rank-two cut constructions.  We prove that rank two is the largest cut-rank that is $F$-hom-free.

\subsection{Cut templates}
Let $U$ be a finite-dimensional vector space over $\mathbb F_2$, and let $U^*=\operatorname{Hom}_{\mathbb F_2}(U,\mathbb F_2)$ be its dual space.  Let $\calC\subseteq U^*\setminus\{0\}$ be a set of nonzero linear forms.
The \emph{cut template} $R(U,\calC)$ is the 3-graph with vertex set
\[
        \{a_c:c\in\calC\}\sqcup U,
\]
where
\[
        \{a_c,u,v\}\in R(U,\calC)
        \quad\Longleftrightarrow\quad
        c(u+v)=1.
\]
The vertices $a_c$ are called \emph{apex vertices}, the vertices of $U$ are called \emph{bottom vertices}, and the \emph{cut-rank} is
\[
        \operatorname{rank}(\calC):=\dim_{\mathbb F_2}\langle\calC\rangle.
\]
Every edge of $R(U,\calC)$ contains exactly one apex vertex and two bottom vertices.

Equivalently, each form $c\in\calC$ is a \emph{graph cut} of the bottom set: since $c(u+v)=c(u)+c(v)$ over $\mathbb F_2$, the condition $c(u+v)=1$ says precisely that $u$ and $v$ lie on opposite sides of the bipartition $U=c^{-1}(0)\sqcup c^{-1}(1)$.  Thus the link graph of the apex $a_c$ on the bottom vertices is the complete bipartite graph $K_{c^{-1}(0),c^{-1}(1)}$.

In ordinary graph-theoretic terms, $R(U,\calC)$ is obtained by placing one apex above each chosen cut and coning over all edges crossing that cut.  The cut-rank is the $\mathbb F_2$-rank of this family of cuts.  In particular, the \emph{full rank-two template} on $\mathbb F_2^2$ has four bottom vertices and three apex link graphs, corresponding to the three bipartitions of these four vertices into two pairs.

\subsection{The rank dichotomy}

We first record a small matrix observation.

\begin{lemma}\label{lem:matrix}
Let $M$ be a $3\times 3$ zero-one matrix such that every \emph{permutation diagonal} has sum one, i.e.
\[
        M_{1,\sigma(1)}+M_{2,\sigma(2)}+M_{3,\sigma(3)}=1
        \qquad\text{for all }\sigma\in S_3.
\]
Then the three 1s of $M$ lie either in a single row or in a single column.
\end{lemma}

\begin{proof}
Each entry of $M$ lies on exactly two permutation diagonals.  Summing the six diagonal equations gives $2\sum_{i,j}M_{ij}=6$, so $M$ has exactly three 1s.  If two 1s were in different rows and different columns, then some permutation diagonal would contain both, a contradiction.
\end{proof}

\begin{lemma}\label{lem:four-by-three}
Let $M$ be a $4\times 3$ zero-one matrix with rows indexed by $a,b,c,d$.  Suppose that the three submatrices on rows $abc$, $abd$ and $acd$ all satisfy the condition in Lemma~\ref{lem:matrix}.  Then one of the following holds:
\begin{enumerate}[label=\emph{(\roman*)},leftmargin=2.2em]
    \item the $a$-row is all 1s and the other three rows are all 0s, or
    \item one column is all 1s and the other two columns are all 0s.
\end{enumerate}
\end{lemma}

\begin{proof}
By Lemma~\ref{lem:matrix}, each of the three displayed $3\times3$ submatrices is of \emph{row type} or \emph{column type}.  If one of them is of column type, then every submatrix sharing two of its rows must be of the same column type; otherwise the two common rows would receive incompatible patterns.  Hence all three submatrices have the same full column of 1s.

It remains to consider the case where all three submatrices are of row type.  The row of 1s must be common to the three submatrices.  Indeed, if the row of 1s in the $abc$ submatrix were, say, $b$ or $c$, then the $acd$ and $abd$ submatrices would force two different non-$a$ rows to be all 1s, contradicting row type in the remaining submatrix.  Thus the common row is $a$.
\end{proof}

\begin{proposition}\label{prop:rank-two-free}
If $\operatorname{rank}(\calC)\le 2$, then $F\nrightarrow R(U,\calC)$. Consequently, every blowup of $R(U,\calC)$ is $F$-free.
\end{proposition}

\begin{proof}
It suffices to prove the statement for the full rank-two template
\[
        R_2:=R(\mathbb F_2^2,\{x,y,x+y\}),
\]
where $x,y$ are the two coordinate forms.  Indeed, every cut template of rank at most two maps homomorphically to $R_2$ by recording the values of two basis cuts.  Suppose, for a contradiction, that there is a homomorphism
\[
        \phi:F=K_4^-\times\Fstar\longrightarrow R_2.
\]

Fix an edge $ijk\in \Fstar$.  Define a $4\times3$ zero-one matrix by
\[
        M_{r,v}=1
        \quad\Longleftrightarrow\quad
        \phi(r,v)\text{ is an apex vertex of }R_2,
        \qquad r\in\{a,b,c,d\},\ v\in\{i,j,k\}.
\]
For each of the three edges $abc,abd,acd$ of $K_4^-$ and for every bijection from this edge to $\{i,j,k\}$, the corresponding triple is an edge of $F$.  Since every edge of $R_2$ contains exactly one apex vertex, the three submatrices on rows $abc$, $abd$ and $acd$ satisfy Lemma~\ref{lem:matrix}.  Lemma~\ref{lem:four-by-three} gives two possibilities for the edge $ijk$: either the $a$-row is the unique all-apex row, or one of the three columns is all-apex.

These two possibilities cannot be mixed on intersecting edges of $\Fstar$, because the column pattern of the common vertex would have to be both $(1,0,0,0)^T$ and either $(1,1,1,1)^T$ or $(0,0,0,0)^T$.  Since the edge-intersection graph of $\Fstar$ is connected, all edges of $\Fstar$ have the same type.

First suppose that all edges have column type.  Then there is a zero-one labelling $s_1,\ldots,s_7$ of $V(\Fstar)$ such that every edge of $\Fstar$ contains exactly one vertex with label 1.  Hence
\[
\begin{aligned}
        s_1+s_2+s_3&=1,
        &s_1+s_2+s_4&=1,
        &s_3+s_4+s_5&=1,\\
        s_1+s_5+s_6&=1,
        &&&s_2+s_5+s_7&=1.
\end{aligned}
\]
The first two equations give $s_3=s_4$.  The third gives $s_3=s_4=0$ and $s_5=1$.  Then the fourth gives $s_1=s_6=0$, the first gives $s_2=1$, and the fifth becomes $1+1+s_7=1$, impossible.

Thus all edges have row type.  For each $i\in V(\Fstar)$ let $c_i\in\{x,y,x+y\}$ be the cut corresponding to the apex vertex $\phi(a,i)$.  For $r\in\{b,c,d\}$, let
\[
        b_i^r\in\mathbb F_2^2
\]
be the bottom point corresponding to $\phi(r,i)$.  If $ijk\in \Fstar$, then the edge condition first gives $c_i(b_j^p+b_k^q)=1$ whenever $p\ne q$.  To see that row differences of the same bottom vertex are annihilated by $c_i$, take distinct rows $p,p'\in\{b,c,d\}$ and choose $q$ different from both.  Applying the edge condition with bottom rows $(p,q)$ and $(p',q)$ gives
\[
        c_i(b_j^p+b_k^q)=c_i(b_j^{p'}+b_k^q)=1,
\]
and hence $c_i(b_j^p+b_j^{p'})=0$.  The same argument applies to the $k$-vertex.  If $p=q$, choose $p'\ne p$; then
\[
        c_i(b_j^p+b_k^p)
        =
        c_i(b_j^{p'}+b_k^p)+c_i(b_j^p+b_j^{p'})
        =
        1+0=1.
\]
Consequently, for every two rows $p,q\in\{b,c,d\}$, possibly equal,
\begin{equation}\label{eq:rank-two-row}
        c_i(b_j^p+b_k^q)=1,
        \qquad
        c_i(b_j^p+b_j^q)=c_i(b_k^p+b_k^q)=0.
\end{equation}

Apply this to the two edges $123$ and $124$.  For instance, applying \eqref{eq:rank-two-row} to $123$ and $124$ with apex $1$ gives
\[
        c_1(b_2^r+b_3^p)=1,
        \qquad
        c_1(b_2^r+b_4^q)=1,
\]
and hence $c_1(b_3^p+b_4^q)=0$.  The same argument with apex $2$ gives the corresponding $c_2$-equation. Thus
\[
        c_1(b_3^p+b_4^q)=c_2(b_3^p+b_4^q)=0
        \qquad\text{for all }p,q,
\]
while the edge $345$, from the apex vertex $5$, gives
\[
        c_5(b_3^p+b_4^q)=1
        \qquad\text{for all }p,q.
\]
We now spell out the consequences of these equations.  If $c_1\ne c_2$, then $\ker c_1\cap\ker c_2=\{0\}$, so all vectors $b_3^p+b_4^q$ would be zero, contradicting $c_5(b_3^p+b_4^q)=1$.  Thus $c_1=c_2$.  Also $c_5\ne c_1$.  Since the system
\[
        c_1(v)=0,
        \qquad
        c_5(v)=1
\]
has a unique solution in $\mathbb F_2^2$, all vectors $b_3^p+b_4^q$ are equal.  Hence the bottom points assigned to vertices $3$ and $4$ are independent of the row indices.  Any invertible linear transformation of $\mathbb F_2^2$ permutes the three nonzero forms $x,y,x+y$ and induces an automorphism of $R_2$, so after a linear change of coordinates we may assume
\[
        c_1=c_2=x.
\]

The bottom points assigned to vertices $1$ and $2$ are also independent of the row indices.  Indeed, row differences for vertex $1$ lie in both $\ker c_2=\ker x$ and $\ker c_5$, and row differences for vertex $2$ lie in both $\ker c_1=\ker x$ and $\ker c_5$; since $c_5\ne x$, both intersections are zero.  Translating all bottom points by a fixed vector preserves $u+v$ in characteristic two and hence induces an automorphism of $R_2$, so we may write $b_1=(0,0)$.  The equations with apices $1$ and $2$ on the edge $123$ give $x(b_1+b_2)=0$, while the equation with apex $3$ gives $c_3(b_1+b_2)=1$; hence $b_2$ is the nonzero vector in $\ker x$.  Thus
\[
        b_1=(0,0),
        \qquad
        b_2=(0,1).
\]
The equations from the apex $1$ in the edges $123$ and $124$ give $x(b_3)=x(b_4)=1$, and the edge $345$ gives $b_3\ne b_4$.  Therefore
\[
        b_3=(1,s),
        \qquad
        b_4=(1,s+1)
\]
for some $s\in\mathbb F_2$.  Since $c_3(b_1+b_2)=c_4(b_1+b_2)=1$, both $c_3$ and $c_4$ have the form
\[
        c_3=p_3x+y,
        \qquad
        c_4=p_4x+y
\]
with $p_3,p_4\in\mathbb F_2$.

Write $b_5^r=(u_r,v_r)$ for $r\in\{b,c,d\}$.  From the edge $345$ we obtain, for every row $r$,
\[
        p_3(1+u_r)+s+v_r=0,
        \qquad
        p_4(1+u_r)+s+v_r=1.
\]
Adding gives $(p_3+p_4)(1+u_r)=1$, and hence $u_r=0$ for every $r$.

Finally, the edges $156$ and $257$ give
\[
        c_6(b_1+b_5^r)=1,
        \qquad
        c_7(b_2+b_5^r)=1
        \qquad\text{for every }r.
\]
But $b_1=(0,0)$, $b_2=(0,1)$, and $b_5^r=(0,v_r)$.  If $v_r=0$, then the first equation evaluates a linear form on the zero vector; if $v_r=1$, then the second equation does.  Both are impossible.
\end{proof}

\begin{proposition}\label{prop:rank-three-contains}
If $\operatorname{rank}(\calC)\ge 3$, then $F\to R(U,\calC)$. Thus Proposition~\ref{prop:rank-two-free} is sharp.
\end{proposition}

\begin{proof}
Choose three independent cuts $c^1,c^2,c^3\in\calC$.  Let $X,Y,Z$ be the coordinate forms on $\mathbb F_2^3$.  The linear map
\[
        \pi:U\to\mathbb F_2^3,
        \qquad
        \pi(u)=(c^1(u),c^2(u),c^3(u))
\]
is surjective, so it has a linear right inverse $s:\mathbb F_2^3\to U$.  Therefore the map sending a bottom vertex $w\in\mathbb F_2^3$ to $s(w)$ and the apex vertices $a_X,a_Y,a_Z$ to $a_{c^1},a_{c^2},a_{c^3}$, respectively, is a homomorphism
\[
        R(\mathbb F_2^3,\{X,Y,Z\})\to R(U,\calC).
\]
Thus it is enough to give a homomorphism into $R(\mathbb F_2^3,\{X,Y,Z\})$.

Set
\[
        d_1=d_2=d_6=d_7=X,
        \qquad
        d_3=d_4=Y,
        \qquad
        d_5=Z,
\]
and choose bottom points
\[
\begin{gathered}
        u_1=000,
        \quad u_2=010,
        \quad u_3=100,
        \quad u_4=101,\\
        u_5=110,
        \quad u_6=001,
        \quad u_7=001.
\end{gathered}
\]
For every edge $ijk\in\Fstar$, the required evaluations are
\[
\begin{array}{c|ccc}
ijk & d_i(u_j+u_k) & d_j(u_i+u_k) & d_k(u_i+u_j)\\
\hline
123 & 1 & 1 & 1\\
124 & 1 & 1 & 1\\
345 & 1 & 1 & 1\\
156 & 1 & 1 & 1\\
257 & 1 & 1 & 1
\end{array}
\]
and hence
\[
        d_i(u_j+u_k)=d_j(u_i+u_k)=d_k(u_i+u_j)=1.
\]
Hence the map
\[
        (a,i)\mapsto a_{d_i},
        \qquad
        (r,i)\mapsto u_i\quad(r\in\{b,c,d\})
\]
is a homomorphism from $K_4^-\times\Fstar$ to $R(\mathbb F_2^3,\{X,Y,Z\})$, and hence, by the lifting observation above, to $R(U,\calC)$.
\end{proof}

\subsection{Crossed blowups and infinite stability}

Let
\[
        R_\times:=R(\mathbb F_2^2,\{x,y\}).
\]
This is the rank-two cut-template version of the \emph{crossed-blowup} construction of Hou, Li, Liu, Mubayi and Zhang \cite{HLLMZ}.  The following elementary construction records the special case needed here.

\begin{theorem}\label{thm:HLLMZ-crossed}
There is a family of $R_\times$-blowups $G_\alpha(n)$, indexed by $\alpha\in(0,1/2)$, such that
\[
        |G_\alpha(n)|=
        \left(\frac1{27}+o(1)\right)n^3
        =\left(\frac29+o(1)\right)\binom n3,
\]
and for every distinct $\alpha,\beta\in(0,1/2)$,
\[
        \dist(G_\alpha(n),G_\beta(n))=\Omega_{\alpha,\beta}(n^3).
\]
\end{theorem}

\begin{proof}
Let the two apex classes corresponding to $a_x$ and $a_y$ be $A_x$ and $A_y$, and let the four bottom classes corresponding to $00,01,10,11\in\mathbb F_2^2$ be $B_{00},B_{01},B_{10},B_{11}$.  Choose their sizes so that they sum to $n$ and satisfy
\[
        |A_x|=|A_y|=\frac n6+O(1),
        \qquad
        |B_{00}|=|B_{11}|=\frac{\alpha n}{3}+O(1),
        \qquad
        |B_{01}|=|B_{10}|=\frac{(1-\alpha)n}{3}+O(1).
\]
Let $G_\alpha(n)$ be the corresponding blowup of $R_\times$.  Each side of each of the two cuts on the bottom vertices has size $n/3+O(1)$, and hence
\[
        |G_\alpha(n)|
        =
        |A_x|\bigl(|B_{00}|+|B_{01}|\bigr)\bigl(|B_{10}|+|B_{11}|\bigr)
        +
        |A_y|\bigl(|B_{00}|+|B_{10}|\bigr)\bigl(|B_{01}|+|B_{11}|\bigr)
        =
        \frac{n^3}{27}+O(n^2).
\]

It remains to prove separation.  For a 3-graph $H$, write
\[
        Q(H):=\sum_{\{u,v\}\in\binom{V(H)}2} d_H(u,v)^2,
\]
where $d_H(u,v)$ is the codegree of $u$ and $v$.  If two $n$-vertex 3-graphs differ in $k$ edges, then changing these edges changes $Q$ by at most $3(2n+1)k$, since changing one edge changes the codegrees of exactly three pairs by $1$.

For $G_\alpha(n)$, the apex--bottom pairs contribute $2n^4/81+o(n^4)$ to $Q$.  Among bottom pairs, pairs differing in exactly one coordinate contribute $\alpha(1-\alpha)n^4/81+o(n^4)$, while pairs differing in both coordinates contribute $\bigl(\alpha^2+(1-\alpha)^2\bigr)n^4/81+o(n^4)$. Thus
\[
        Q(G_\alpha(n))
        =
        \left(\frac{3-\alpha+\alpha^2}{81}+o(1)\right)n^4.
\]
The function $3-\alpha+\alpha^2$ is strictly decreasing on $(0,1/2)$. Hence, for distinct $\alpha,\beta\in(0,1/2)$, the quantities $Q(G_\alpha(n))$ and $Q(G_\beta(n))$ differ by $\Omega_{\alpha,\beta}(n^4)$. Since $Q$ is invariant under relabelling, the preceding Lipschitz bound implies $\dist(G_\alpha(n),G_\beta(n))=\Omega_{\alpha,\beta}(n^3)$.
\end{proof}

\begin{proof}[Proof of Theorem~\ref{thm:main}]
Corollary~\ref{cor:exact-density} gives $\pi(F)=2/9$.  It remains to prove infinite stability.  By Proposition~\ref{prop:rank-two-free}, every $G_\alpha(n)$ in Theorem~\ref{thm:HLLMZ-crossed} is $F$-free, and by Corollary~\ref{cor:exact-density} these graphs are asymptotically extremal for $F$.

Suppose, for a contradiction, that $\xi(F)\le t$. Choose distinct parameters $\alpha_1,\ldots,\alpha_{t+1}\in(0,1/2)$.  By Theorem~\ref{thm:HLLMZ-crossed}, the finitely many pairwise separation constants have a positive minimum $c>0$. Thus, for all sufficiently large $n$,
\[
        \dist(G_{\alpha_i}(n),G_{\alpha_j}(n))\ge c n^3
        \qquad\text{for all }i\ne j.
\]
Set $\delta=c/3$.  By $t$-stability, fix the corresponding template graphs $H_1(m),\ldots,H_t(m)$ for each $m$. For this value of $\delta$, there are $\epsilon>0$ and $N_0$ such that every $n$-vertex $F$-free graph with more than $(1-\epsilon)\ex(n,F)$ edges is within $\delta n^3$ edits of one of the $H_\ell(n)$.  For all sufficiently large $n$, each $G_{\alpha_i}(n)$ satisfies this edge condition, since it is asymptotically extremal.  By the pigeonhole principle, two of the $t+1$ graphs $G_{\alpha_i}(n)$ and $G_{\alpha_j}(n)$ are both within $\delta n^3$ edits of the same template.  Hence, by the triangle inequality for edit distance, their edit distance is at most $2\delta n^3<c n^3$, a contradiction.  Therefore $\xi(F)=\infty$.
\end{proof}

\section*{Acknowledgements}
H.L. was supported by the National Natural Science Foundation of China (12501487), by China Scholarship Council
and IBS-R029-C4.
X.L. was supported by the Excellent Young Talents Program (Overseas) of the National Natural Science Foundation of China.

\section*{Declaration on the use of AI}
The authors used generative AI tools to assist in discussing proof strategies, checking proofs, and improving exposition. All mathematical arguments, results, and conclusions were reviewed and verified by the authors.

\bibliographystyle{abbrv}
\bibliography{Turan}

@article {BaloghClemenLuo,
    AUTHOR = {Balogh, J\'{o}zsef and Clemen, Felix Christian and Luo,
              Haoran},
     TITLE = {Non-degenerate hypergraphs with exponentially many extremal
              constructions},
   JOURNAL = {J. Combin. Theory Ser. B},
  FJOURNAL = {Journal of Combinatorial Theory. Series B},
    VOLUME = {175},
      YEAR = {2025},
     PAGES = {1--28},
      ISSN = {0095-8956,1096-0902},
   MRCLASS = {05C35 (05C65)},
  MRNUMBER = {4919811},
MRREVIEWER = {Yuval\ Wigderson},
       DOI = {10.1016/j.jctb.2025.06.001},
       URL = {https://doi.org/10.1016/j.jctb.2025.06.001},
}

@article {BaloghMubayi,
    AUTHOR = {Balogh, J\'{o}zsef and Mubayi, Dhruv},
     TITLE = {Almost all triangle-free triple systems are tripartite},
   JOURNAL = {Combinatorica},
  FJOURNAL = {Combinatorica. An International Journal on Combinatorics and
              the Theory of Computing},
    VOLUME = {32},
      YEAR = {2012},
    NUMBER = {2},
     PAGES = {143--169},
      ISSN = {0209-9683,1439-6912},
   MRCLASS = {05C65 (05D40)},
  MRNUMBER = {2927636},
MRREVIEWER = {A.\ Rosa},
       DOI = {10.1007/s00493-012-2657-4},
       URL = {https://doi.org/10.1007/s00493-012-2657-4},
}

@article {Bollobas,
    AUTHOR = {Bollob\'{a}s, B\'{e}la},
     TITLE = {Three-graphs without two triples whose symmetric difference is
              contained in a third},
   JOURNAL = {Discrete Math.},
  FJOURNAL = {Discrete Mathematics},
    VOLUME = {8},
      YEAR = {1974},
     PAGES = {21--24},
      ISSN = {0012-365X,1872-681X},
   MRCLASS = {05C99},
  MRNUMBER = {345869},
       DOI = {10.1016/0012-365X(74)90105-8},
       URL = {https://doi.org/10.1016/0012-365X(74)90105-8},
}

@incollection {BrownTuran34,
    AUTHOR = {Brown, W. G.},
     TITLE = {On an open problem of {P}aul {T}ur\'{a}n concerning
              {$3$}-graphs},
 BOOKTITLE = {Studies in pure mathematics},
     PAGES = {91--93},
 PUBLISHER = {Birkh\"{a}user, Basel},
      YEAR = {1983},
      ISBN = {3-7643-1288-2},
   MRCLASS = {05C35},
  MRNUMBER = {820213},
}

@article {BrownSimonovits,
    AUTHOR = {Brown, W. G. and Simonovits, M.},
     TITLE = {Digraph extremal problems, hypergraph extremal problems, and
              the densities of graph structures},
   JOURNAL = {Discrete Math.},
  FJOURNAL = {Discrete Mathematics},
    VOLUME = {48},
      YEAR = {1984},
    NUMBER = {2-3},
     PAGES = {147--162},
      ISSN = {0012-365X,1872-681X},
   MRCLASS = {05C35},
  MRNUMBER = {737261},
MRREVIEWER = {G.\ W.\ Peck},
       DOI = {10.1016/0012-365X(84)90178-X},
       URL = {https://doi.org/10.1016/0012-365X(84)90178-X},
}

@article {Culik1958,
    AUTHOR = {\v{C}ul\'{\i}k, Karel},
     TITLE = {Zur {T}heorie der {G}raphen},
   JOURNAL = {\v{C}asopis P\v{e}st. Mat.},
  FJOURNAL = {\v{C}eskoslovensk\'{a} Akademie V\v{e}d. \v{C}asopis Pro
              P\v{e}stov\'{a}n\'{\i} Matematiky},
    VOLUME = {83},
      YEAR = {1958},
     PAGES = {133--155},
      ISSN = {0528-2195},
   MRCLASS = {55.00 (04.00)},
  MRNUMBER = {97805},
MRREVIEWER = {Gert\ Sabidussi},
}

@article {DorflerWaller,
    AUTHOR = {D\"{o}rfler, W. and Waller, D. A.},
     TITLE = {A category-theoretical approach to hypergraphs},
   JOURNAL = {Arch. Math. (Basel)},
  FJOURNAL = {Archiv der Mathematik},
    VOLUME = {34},
      YEAR = {1980},
    NUMBER = {2},
     PAGES = {185--192},
      ISSN = {0003-889X,1420-8938},
   MRCLASS = {05C65 (18B99)},
  MRNUMBER = {583768},
MRREVIEWER = {A.\ Pultr},
       DOI = {10.1007/BF01224952},
       URL = {https://doi.org/10.1007/BF01224952},
}

@article {ErdosStone,
    AUTHOR = {Erd\H{o}s, P. and Stone, A. H.},
     TITLE = {On the structure of linear graphs},
   JOURNAL = {Bull. Amer. Math. Soc.},
  FJOURNAL = {Bulletin of the American Mathematical Society},
    VOLUME = {52},
      YEAR = {1946},
     PAGES = {1087--1091},
      ISSN = {0002-9904},
   MRCLASS = {56.0X},
  MRNUMBER = {18807},
MRREVIEWER = {H.\ S. M. Coxeter},
       DOI = {10.1090/S0002-9904-1946-08715-7},
       URL = {https://doi.org/10.1090/S0002-9904-1946-08715-7},
}

@article {ErdosSimonovitsLimit,
    AUTHOR = {Erd\H{o}s, P. and Simonovits, M.},
     TITLE = {A limit theorem in graph theory},
   JOURNAL = {Studia Sci. Math. Hungar.},
  FJOURNAL = {Studia Scientiarum Mathematicarum Hungarica. Combinatorics,
              Geometry and Topology (CoGeTo)},
    VOLUME = {1},
      YEAR = {1966},
     PAGES = {51--57},
      ISSN = {0081-6906,1588-2896},
   MRCLASS = {05.40},
  MRNUMBER = {205876},
MRREVIEWER = {W.\ Moser},
}

@article {FranklFuredi,
    AUTHOR = {Frankl, Peter and F\"{u}redi, Zolt\'{a}n},
     TITLE = {A new generalization of the {E}rd{\H{o}}s--{K}o--{R}ado theorem},
   JOURNAL = {Combinatorica},
  FJOURNAL = {Combinatorica. An International Journal of the J\'{a}nos
              Bolyai Mathematical Society},
    VOLUME = {3},
      YEAR = {1983},
    NUMBER = {3-4},
     PAGES = {341--349},
      ISSN = {0209-9683},
   MRCLASS = {05A05 (05B30)},
  MRNUMBER = {729787},
MRREVIEWER = {\c{S}erban\ Buze\c{t}eanu},
       DOI = {10.1007/BF02579190},
       URL = {https://doi.org/10.1007/BF02579190},
}

@article {Frohmader,
    AUTHOR = {Frohmader, Andrew},
     TITLE = {More constructions for {T}ur\'{a}n's {$(3,4)$}-conjecture},
   JOURNAL = {Electron. J. Combin.},
  FJOURNAL = {Electronic Journal of Combinatorics},
    VOLUME = {15},
      YEAR = {2008},
    NUMBER = {1},
     PAGES = {Research Paper 137, 23},
      ISSN = {1077-8926},
   MRCLASS = {05C65},
  MRNUMBER = {2465761},
MRREVIEWER = {Yi\ Zhao},
       DOI = {10.37236/861},
       URL = {https://doi.org/10.37236/861},
}

@article {HellmuthOstermeierStadler,
    AUTHOR = {Hellmuth, Marc and Ostermeier, Lydia and Stadler, Peter F.},
     TITLE = {A survey on hypergraph products},
   JOURNAL = {Math. Comput. Sci.},
  FJOURNAL = {Mathematics in Computer Science},
    VOLUME = {6},
      YEAR = {2012},
    NUMBER = {1},
     PAGES = {1--32},
      ISSN = {1661-8270,1661-8289},
   MRCLASS = {05C65 (05C76 68R10)},
  MRNUMBER = {2910504},
       DOI = {10.1007/s11786-012-0109-6},
       URL = {https://doi.org/10.1007/s11786-012-0109-6},
}

@article {HLLMZ,
    AUTHOR = {Hou, Jianfeng and Li, Heng and Liu, Xizhi and Mubayi, Dhruv
              and Zhang, Yixiao},
     TITLE = {Hypergraphs with infinitely many extremal constructions},
   JOURNAL = {Discrete Anal.},
  FJOURNAL = {Discrete Analysis},
      YEAR = {2023},
      NOTE = {Paper No. 18, 34 pp.},
      ISSN = {2397-3129},
   MRCLASS = {05C35 (05C65)},
  MRNUMBER = {4684860},
MRREVIEWER = {Mark\ Rowland\ Budden},
       DOI = {10.19086/da.88508},
       URL = {https://doi.org/10.19086/da.88508},
}

@incollection {KeevashSurvey,
    AUTHOR = {Keevash, Peter},
     TITLE = {Hypergraph {T}ur\'{a}n problems},
 BOOKTITLE = {Surveys in combinatorics 2011},
    SERIES = {London Math. Soc. Lecture Note Ser.},
    VOLUME = {392},
     PAGES = {83--139},
 PUBLISHER = {Cambridge Univ. Press, Cambridge},
      YEAR = {2011},
      ISBN = {978-1-107-60109-3},
   MRCLASS = {05-02 (05C65)},
  MRNUMBER = {2866732},
}

@article {KeevashMubayi,
    AUTHOR = {Keevash, Peter and Mubayi, Dhruv},
     TITLE = {Stability theorems for cancellative hypergraphs},
   JOURNAL = {J. Combin. Theory Ser. B},
  FJOURNAL = {Journal of Combinatorial Theory. Series B},
    VOLUME = {92},
      YEAR = {2004},
    NUMBER = {1},
     PAGES = {163--175},
      ISSN = {0095-8956,1096-0902},
   MRCLASS = {05C65 (05C35)},
  MRNUMBER = {2078500},
MRREVIEWER = {Jen\H{o}\ Lehel},
       DOI = {10.1016/j.jctb.2004.05.003},
       URL = {https://doi.org/10.1016/j.jctb.2004.05.003},
}

@article {Kostochka,
    AUTHOR = {Kostochka, A. V.},
     TITLE = {A class of constructions for {T}ur\'{a}n's
              {$(3,\,4)$}-problem},
   JOURNAL = {Combinatorica},
  FJOURNAL = {Combinatorica. An International Journal of the J\'{a}nos
              Bolyai Mathematical Society},
    VOLUME = {2},
      YEAR = {1982},
    NUMBER = {2},
     PAGES = {187--192},
      ISSN = {0209-9683},
   MRCLASS = {05C35 (05C65)},
  MRNUMBER = {685045},
MRREVIEWER = {K.\ Vesztergombi},
       DOI = {10.1007/BF02579317},
       URL = {https://doi.org/10.1007/BF02579317},
}

@misc {KralKucerakLamaisonTardos,
    AUTHOR = {Kr\'{a}l', Daniel and Ku\v{c}er\'{a}k, Filip and Lamaison,
              Ander and Tardos, G\'{a}bor},
     TITLE = {Uniform {T}ur\'{a}n density---palette classification},
      NOTE = {arXiv:2505.17325},
      YEAR = {2025},
       URL = {https://arxiv.org/abs/2505.17325},
}

@misc {LamaisonPalette,
    AUTHOR = {Lamaison, Ander},
     TITLE = {Palettes determine uniform {T}ur\'{a}n density},
      NOTE = {arXiv:2408.09643},
      YEAR = {2024},
       URL = {https://arxiv.org/abs/2408.09643},
}

@article {LiuMubayiNoStability,
    AUTHOR = {Liu, Xizhi and Mubayi, Dhruv},
     TITLE = {A hypergraph {T}ur\'{a}n problem with no stability},
   JOURNAL = {Combinatorica},
  FJOURNAL = {Combinatorica. An International Journal on Combinatorics and
              the Theory of Computing},
    VOLUME = {42},
      YEAR = {2022},
    NUMBER = {3},
     PAGES = {433--462},
      ISSN = {0209-9683,1439-6912},
   MRCLASS = {05C35 (05D05)},
  MRNUMBER = {4482095},
       DOI = {10.1007/s00493-021-4561-2},
       URL = {https://doi.org/10.1007/s00493-021-4561-2},
}

@article {LiuMubayiReiherUnified,
    AUTHOR = {Liu, Xizhi and Mubayi, Dhruv and Reiher, Christian},
     TITLE = {A unified approach to hypergraph stability},
   JOURNAL = {J. Combin. Theory Ser. B},
  FJOURNAL = {Journal of Combinatorial Theory. Series B},
    VOLUME = {158},
      YEAR = {2023},
    NUMBER = {part 2},
     PAGES = {36--62},
      ISSN = {0095-8956,1096-0902},
   MRCLASS = {05C35 (05C65)},
  MRNUMBER = {4484827},
MRREVIEWER = {Jian\ Wang},
       DOI = {10.1016/j.jctb.2022.08.008},
       URL = {https://doi.org/10.1016/j.jctb.2022.08.008},
}

@article {LiuMubayiReiherMany,
    AUTHOR = {Liu, Xizhi and Mubayi, Dhruv and Reiher, Christian},
     TITLE = {Hypergraphs with many extremal configurations},
   JOURNAL = {Israel J. Math.},
  FJOURNAL = {Israel Journal of Mathematics},
    VOLUME = {271},
      YEAR = {2026},
    NUMBER = {1},
     PAGES = {1--38},
      ISSN = {0021-2172,1565-8511},
   MRCLASS = {05C35 (05C65)},
  MRNUMBER = {5014038},
       DOI = {10.1007/s11856-025-2792-4},
       URL = {https://doi.org/10.1007/s11856-025-2792-4},
}

@article {LiuPikhurko,
    AUTHOR = {Liu, Xizhi and Pikhurko, Oleg},
     TITLE = {Finite hypergraph families with rich extremal {T}ur\'{a}n
              constructions via mixing patterns},
   JOURNAL = {Forum Math. Sigma},
  FJOURNAL = {Forum of Mathematics. Sigma},
    VOLUME = {13},
      YEAR = {2025},
     PAGES = {Paper No. e53, 49},
      ISSN = {2050-5094},
   MRCLASS = {05D05 (05C35 05C65)},
  MRNUMBER = {4874849},
       DOI = {10.1017/fms.2025.12},
       URL = {https://doi.org/10.1017/fms.2025.12},
}

@article {MubayiTriangleFree,
    AUTHOR = {Mubayi, Dhruv},
     TITLE = {Structure and stability of triangle-free set systems},
   JOURNAL = {Trans. Amer. Math. Soc.},
  FJOURNAL = {Transactions of the American Mathematical Society},
    VOLUME = {359},
      YEAR = {2007},
    NUMBER = {1},
     PAGES = {275--291},
      ISSN = {0002-9947,1088-6850},
   MRCLASS = {05C35 (05C65)},
  MRNUMBER = {2247891},
MRREVIEWER = {J\'{o}zsef\ Balogh},
       DOI = {10.1090/S0002-9947-06-04009-8},
       URL = {https://doi.org/10.1090/S0002-9947-06-04009-8},
}

@article {PikhurkoGeneralizedTriangle,
    AUTHOR = {Pikhurko, Oleg},
     TITLE = {An exact {T}ur\'{a}n result for the generalized triangle},
   JOURNAL = {Combinatorica},
  FJOURNAL = {Combinatorica. An International Journal on Combinatorics and
              the Theory of Computing},
    VOLUME = {28},
      YEAR = {2008},
    NUMBER = {2},
     PAGES = {187--208},
      ISSN = {0209-9683,1439-6912},
       DOI = {10.1007/s00493-008-2187-2},
       URL = {https://doi.org/10.1007/s00493-008-2187-2},
}

@inproceedings {SimonovitsStability,
    AUTHOR = {Simonovits, M.},
     TITLE = {A method for solving extremal problems in graph theory,
              stability problems},
 BOOKTITLE = {Theory of {G}raphs ({P}roc. {C}olloq., {T}ihany, 1966)},
     PAGES = {279--319},
 PUBLISHER = {Academic Press, New York-London},
      YEAR = {1968},
   MRCLASS = {05.55},
  MRNUMBER = {233735},
MRREVIEWER = {J.\ W.\ Moon},
}

@article {Turan1941,
    AUTHOR = {Tur\'{a}n, Paul},
     TITLE = {Eine {E}xtremalaufgabe aus der {G}raphentheorie},
   JOURNAL = {Mat. Fiz. Lapok},
  FJOURNAL = {Matematikai \'{e}s Fizikai Lapok},
    VOLUME = {48},
      YEAR = {1941},
     PAGES = {436--452},
      ISSN = {0302-7317},
   MRCLASS = {56.0X},
  MRNUMBER = {18405},
MRREVIEWER = {P.\ Erd\H{o}s},
}

@article {Weichsel1962,
    AUTHOR = {Weichsel, Paul M.},
     TITLE = {The {K}ronecker product of graphs},
   JOURNAL = {Proc. Amer. Math. Soc.},
  FJOURNAL = {Proceedings of the American Mathematical Society},
    VOLUME = {13},
      YEAR = {1962},
     PAGES = {47--52},
      ISSN = {0002-9939,1088-6826},
   MRCLASS = {55.10},
  MRNUMBER = {133816},
MRREVIEWER = {W.\ Moser},
       DOI = {10.2307/2033769},
       URL = {https://doi.org/10.2307/2033769},
}

@article {ZhangHouLi2Stable,
    AUTHOR = {Zhang, Yixiao and Hou, Jianfeng and Li, Heng},
     TITLE = {A 2-stable family of triple systems},
   JOURNAL = {Electron. J. Combin.},
  FJOURNAL = {Electronic Journal of Combinatorics},
    VOLUME = {31},
      YEAR = {2024},
    NUMBER = {2},
     PAGES = {Paper No. 2.3, 23},
      ISSN = {1077-8926},
   MRCLASS = {05C65 (05C35 05D05)},
  MRNUMBER = {4734462},
       DOI = {10.37236/11701},
       URL = {https://doi.org/10.37236/11701},
}

@article {Zykov1949,
    AUTHOR = {Zykov, A. A.},
     TITLE = {On some properties of linear complexes},
   JOURNAL = {Mat. Sbornik N.S.},
  FJOURNAL = {Mat. Sbornik N.S.},
    VOLUME = {24(66)},
      YEAR = {1949},
     PAGES = {163--188},
   MRCLASS = {56.0X},
  MRNUMBER = {35428},
MRREVIEWER = {W.\ T.\ Tutte},
}

\end{document}